\documentclass[letter,11pt]{amsart}
\usepackage{amssymb} 
\usepackage{graphicx}
\usepackage{color}
\usepackage[font=footnotesize]{caption}
\usepackage{mathtools}

\usepackage{hyperref}
\usepackage{float}
\usepackage{calc}
\def\thetitle{{Constructing pseudo-Anosov maps with given dilatations}}
\hypersetup{
   pdftitle=   \thetitle,
   pdfauthor=  {Hyungryul Baik, Ahmad Rafiqi, and Chenxi Wu}
}

\newtheorem{thm}{Theorem}
\newtheorem*{thm*}{Theorem}

\theoremstyle{remark}

\newtheorem*{rem}{Remark}
\usepackage{tikz}
\usetikzlibrary{calc,decorations.markings}
\usetikzlibrary{shapes,snakes}

\theoremstyle{definition}
\newtheorem*{defn*}{Definition}

\DeclarePairedDelimiter\floor{\lfloor}{\rfloor}

\begin{document}

\title\thetitle

\author{Hyungryul Baik}
\address{Mathematisches Institut,
Rheinische Friedrich-Wilhelms-Universit\"{a}t Bonn,
Endenicher Allee 60,
53115 Bonn, Germany}
\email{baik@math.uni-bonn.de}

\author{Ahmad Rafiqi}
\address{Department of Mathematics, Cornell University, 111 Malott Hall, Ithaca, NY 14853, USA}
\email{ar776@cornell.edu}

\author{Chenxi Wu}
\address{Department of Mathematics, Cornell University, 105 Malott Hall, Ithaca, NY 14853, USA}
\email{cw538@cornell.edu}
\date{\today}

\keywords{pseudo-Anosov, dilatation, Perron}
\begin{abstract}
In this paper, we give sufficient conditions for a Perron number, given as the leading eigenvalue of an aperiodic matrix, to be a pseudo-Anosov dilatation of a compact surface. We give an explicit construction of the surface and the map when the sufficient condition is met. 
\end{abstract}

\maketitle


\section{Introduction}

The question of which positive algebraic integers can be realized as the dilatation constant of a pseudo-Anosov surface diffeomorphism has a long history (for instance, see \cite{Farb}). A well-known necessary condition is that the number must be strictly greater in absolute value than all its Galois conjugates. 

\begin{defn*}
An algebraic integer (i.e. the root of a polynomial over the integers with leading coefficient $1$) is called \emph{\textbf{Perron}} if it is positive and strictly greater in absolute value than all its Galois conjugates. A non-negative matrix $A$ is \emph{\textbf{aperiodic}} if all the entries of $A^n$ are positive for some $n>0$.
\end{defn*}

\begin{thm}
[Perron-Frobenius] Every aperiodic matrix, $A$, has a positive leading eigenvalue, $\lambda$, and a corresponding positive eigenvector. 
\end{thm}

If $A$ consists only of integers, then the leading eigenvalue $\lambda$ is Perron. A Perron number $\lambda$ all of whose Galois conjugates lie in the annulus $\{z\in \mathbb{C} : \frac{1}{\lambda} \le |z| \le \lambda \}$ is called \emph{\textbf{bi-Perron}}. Dilatations of pseudo-Anosov maps are known to be bi-Perron, as shown by Fried \cite{Fr}. Doug Lind proved a converse to the integer version of Thm $1$ \cite{Li}. 

\begin{thm}
[Lind] Every Perron number, $\lambda$, is the leading eigenvalue of an integer aperiodic matrix.
\end{thm}

Thurston proved a converse to Theorem $1$ more relevant to us. \cite{Th} 

\begin{thm}
[Thurston] Every Perron number, $\lambda$, is the leading eigenvalue of a non-negative integer matrix, $A$, satisfying:
\begin{itemize}
\item[(i)] In each column of $A$, the non-zero entries form one consecutive block,
\item[(ii)] There is a map $\phi:\{0,1,...,n\}\to \{0,1,...,n\}$ such that the entry $A_{ij}$ is odd if and only if $\min\{\phi(j-1),\phi(j)\} < i \le \max\{\phi(j-1),\phi(j)\}$.
\end{itemize}
\end{thm}

\begin{defn*}
A non-negative integer matrix, $A$, satisfying properties $(i)$ and $(ii)$ above will be called an \emph{\textbf{odd-block}} matrix. We write $\phi(i)$ as $\phi_i$.
\end{defn*}

Thurston's example of an odd-block matrix is:

{\footnotesize
\[ \left( \begin{array}{ccccccc}
5 & 6 & 0 & 0 & 0 & 0 & 0\\
1 & 2 & 0 & 0 & 0 & 4 & 0\\
3 & 5 & 1 & 0 & 1 & 2 & 1\\
8 & 4 & 1 & 0 & 7 & 4 & 1\\
2 & 0 & 1 & 3 & 0 & 6 & 1\\
0 & 0 & 0 & 0 & 0 & 0 & 1\\
0 & 0 & 0 & 0 & 0 & 0 & 1\\
\end{array} \right)\]
}
 
{\footnotesize Here $(\phi_i)_{i=0}^7 = (0, 3, 2, 5, 4, 2, 2, 7)$. For instance, since $\phi_1=3$ and $\phi_2=2$, only the $3^{rd}$ entry of the second column is odd.}\\

The question we're trying to address is: Given a bi-Perron number $\lambda$, can we construct a compact surface $S_g$, and a pseudo-Anosov homeomorphism $\psi:S_g\to S_g$, whose dilatation factor is $\lambda$?\\

Our strategy is to enforce further conditions on an odd-block matrix, $M$, that are sufficient to guarantee that its leading eigenvalue is a pseudo-Anosov dilatation. We will first describe a general construction, possibly yielding a surface of infinite type, and then prove that our conditions are sufficient to ensure the surface constructed is of finite type.\\

Based on whether we want to construct an orientation preserving, or reversing map, different conditions are forced on the matrix. To keep track of this, we define a parameter $\boldsymbol{\epsilon \in \{1,-1\}}$. We will talk about the pair $(M,\epsilon)$ or just the matrix $M$ satisfying some conditions depending on whether it matters if we preserve/reverse orientation or not.\\

For now, we are only able to deal with matrices consisting entirely of $0$'s and $1$'s. Presumably one can generalize our construction to matrices with larger entries, but we don't know yet how to do so.\\

\begin{thm} 
If $\lambda$ is the leading eigenvalue of a nonsingular, aperiodic, odd-block, $\{0,1\}$-matrix $M$, such that $M$ satisfies the \emph{one-sided} condition, and $(M,\epsilon)$ satisfies the \emph{alignment} condition (defined in the section $3$), then\\ 
\begin{itemize}
\item[(i)] $\lambda$ is the dilatation of a pseudo-Anosov map, $\psi$, of a compact surface, $S_g$, of genus $g\le\frac{1}{2}\dim(M)$. \\
\item[(ii)] $\psi$ is orientation preserving (resp. reversing) when $\epsilon$ is 1 (resp. -1). \\
\item[(iii)] When $g=\frac{1}{2}\dim(M)$ a basis for $H_1(S_g)$ may be chosen so that the action of $\psi$ on $H_1(S_g)$ is given by $M$. $\chi(M)$ is palindromic or anti-palindromic in this case (resp. whether $\epsilon = 1$ or $-1$).
\end{itemize}
\end{thm}

In section $2$, following \cite{Th}, we define an intermediate interval map associated to an odd-block matrix. In section $3$, we define further conditions on $M$ that we call \emph{one-sided} and \emph{alignment} conditions. In section $4$ we describe a construction of $S_g$ and $\psi$. In section $5$, we prove theorem 4. \\

\section{Interval map $\boldsymbol{h}$ associated to $M$}

Throughout sections $2, 3$ and $4$, let $M_{n\times n}$ be a nonsingular, aperiodic, odd-block matrix consisting entirely of $0$'s and $1$'s. Further, let $\epsilon\in\{-1,1\}$ be fixed. Let $\lambda$ be the positive leading eigenvalue of $M$. Recall that $M$, being odd-block, comes with a map $\phi:\{0,1,...,n\}\to\{0,1,...,n\}$. Before we describe our conditions on $M$, we need to define a convenient map $\boldsymbol{h_M=h:[0,1]\to [0,1]}$ associated to $M$.\\

Say $\overrightarrow{w}=(w_1,w_2,...,w_n)$ is a positive eigenvector of $M^T$ corresponding to $\lambda$. Normalize $\overrightarrow{w}$ so that $\sum_{i=1}^nw_i=1$. Partition $[0,1]$ as \{{\small $0=x_0<x_1<...<x_n=1$}\} such that $x_i-x_{i-1} = w_i$. Let $\boldsymbol{I_i} := (x_{i-1},x_i)$. Now define $h$ by sending $x_i$ to $x_{\phi_i}$, and connecting linearly in between. More precisely, for $x\in I_i$, define $h(x) = x_{\phi_{i-1}} + \frac{x_{\phi_i}-x_{\phi_{i-1}}}{x_i-x_{i-1}}(x-x_{i-1})$.
{\footnotesize
\begin{figure}[H]
\[ \left( \begin{array}{cccccc}
0 & 0 & 0 & 0 & 0 & 1\\
1 & 1 & 0 & 0 & 0 & 1\\
0 & 1 & 0 & 0 & 0 & 1\\
0 & 0 & 1 & 0 & 0 & 1\\
0 & 0 & 1 & 0 & 1 & 0\\
0 & 0 & 0 & 1 & 1 & 0\\
\end{array} \right) \]
\begin{center}
\includegraphics[scale=0.7]{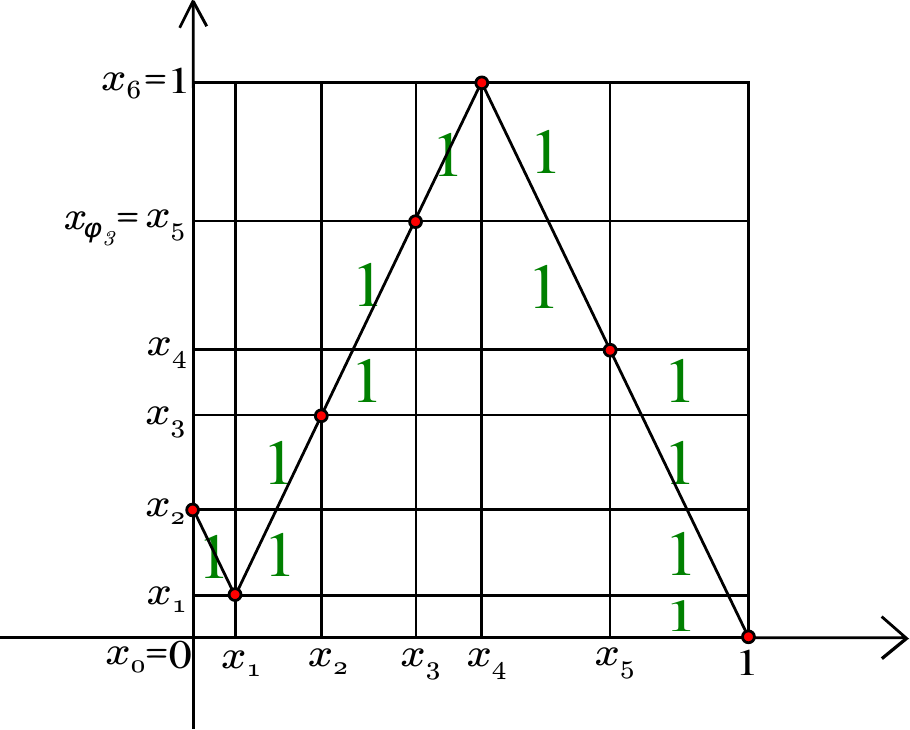}
\end{center}
\caption{\small An example of $M$ and the associated map $h$. Here $(\phi_i)_{i=0}^6 = (2, 1, 3, 5, 6, 4, 0),\, \chi(M) = (x^2-x+1)*(x^4-x^3-2x^2-x+1),\, \lambda \approx 2.081$. $h$ ``resembles'' the reflection of $M$ vertically. }
\end{figure}
}

$h$ is postcritically finite since its postcritical set is contained in the set $X=\{x_0,...,x_n\}$ on which it acts by changing subscripts via $\phi$. This is a permutation: otherwise some $x_i$ will not be in the image of $h$, so no column of $1$s will begin or end between rows $i$ and $i+1$ of $M$, which implies these rows will be the same, leading to $\det(M)=0$.\\

$h$ has constant slope, $\lambda$, in absolute value. This follows precisely from setting interval widths according to the eigenvector $\overrightarrow{w}$ of $M^T$, corresponding to $\lambda$. Finally, we mention that the \emph{odd-block} property of $M$, found by Thurston, corresponds to \emph{continuity} of $h$ \cite[Prop. 3.1]{Th}.

\section{The \textbf{one-sided} and \textbf{alignment} conditions}
Call an $x_i\in X$ a \emph{\textbf{noncritical}} point of $h$ if it is not a (topological) critical point.
\vspace{7mm}
\begin{defn*}
$M$ is said to satisfy the \textbf{\emph{one-sided condition}} if: For any $x_i\in X$, the horizontal line through $h(x_i)$ does not intersect the graph of $h$ on both sides of $x_i$. i.e. \[\forall\, x_i\in X, \,\,\nexists\,\, x, y \in [0,1], \text{ with } x<x_i<y \text{ such that } h(x)=h(x_i)=h(y).\] 
\end{defn*}
\vspace{7mm}
With $M$ satisfying this condition, if the horizontal line through $x_i$, intersects the graph of $h$ on the right (left respectively) of $x_{\phi^{-1}(i)}$, define the \emph{\textbf{alignment at}} $\boldsymbol{x_i:=\alpha(i)} := 1$ ($-1$ respectively). The alignment function $\alpha$ is thus defined on a subset of $\{1,2,...,n-1\}$. It may not be defined at some $x_i$.\\

Recall that $\epsilon=\pm1$ keeps track of whether we're trying to construct an orientation preserving or reversing map.
\vspace{7mm}
\begin{defn*}
$(M,\epsilon)$ is said to satisfy the \textbf{\emph{alignment condition}} if the alignment function $\alpha$ can be extended to a function $\alpha:\{1,2,...,n-1\}\to\{1,-1\}$ satisfying:
\begin{itemize}
\item[(a)] For critical $x_i, \,\,\alpha(i) = \begin{cases}
      \hfill -\epsilon \hfill & \text{ if $x_i$ is a local max.} \\ 
      \hfill +\epsilon \hfill & \text{ if $x_i$ is a local min.} \\
  		\end{cases}$
\item[(b)] For noncritical $x_i,\,\,\alpha(i) = \begin{cases}
      \hfill +\epsilon\,\alpha(\phi_i) \hfill & \text{ if $h'(x_i)>0$.}\\
      \hfill -\epsilon\,\alpha(\phi_i) \hfill & \text{ if $h'(x_i)<0$.}\\
  		\end{cases}$
\end{itemize}
\end{defn*}
\vspace{7mm}
We will motivate these definitions in the next section where we describe our main construction, for now we merely say that our construction can be adapted for matrices that don't satisfy one-sided and alignment conditions, but in such cases, the surface constructed is of infinite type. It is these conditions that ensure finiteness of the surface we construct. \\

Our construction without one-sided and alignment conditions yields a family of generalized pseudo-Anosov maps in the sense of de Carvalho-Hall  \cite{CH}. \cite{CH} constructed an infinite family of such maps from unimodal piecewise linear maps of the interval, and our construction is for multimodal maps. \\

We may assume without loss of generality that each cycle of the permutation $\phi$ contains a critical point of $h$. This is because if there was a cycle consisting entirely of noncritical points, we could simply remove them from our partition of $[0,1]$ and get a smaller matrix with the same leading eigenvalue. The smaller matrix would still consist of only $0$s and $1$s, as these points are noncritical. Finally, as this process doesn't change $h$, the smaller matrix would also satisfy the one-sided and alignment conditions. \\

With this assumption, $X=\{x_0,...,x_n\}$ is in fact the postcritical set of $h$. Since we don't lose any eigenvalues $\lambda$ this way, we make this assumption in what follows. For this, define a $\emph{\textbf{minimally odd-block}}$ matrix $M$, as an odd-block matrix that satisfies this assumption.

\section{The construction}

In this section, we construct a surface and a homeomorphism corresponding to $(M,\epsilon)$. The one-sided and alignment conditions will imply that the surface is compact and the homeomorphism pseudo-Anosov.\\

$M$, and its transpose, $M^T$, both have positive eigenvectors (say, $\overrightarrow{v} = (v_i)$ and $\overrightarrow{w} = (w_i)$), corresponding to their leading eigenvalue $\lambda$. Normalize these so that $\sum_{i=1}^nw_i=\sum_{i=1}^nv_i=1$.\\

Partition $[0,1]\times[0,1]$ into an $n\times n$ grid, with widths given by the entries of $\overrightarrow{v}$ (going left to right), and heights given by the entries of $\overrightarrow{w}$ (going from top to bottom). The vertical subintervals are the $I_i$. Shade the $\{i,j\}^{th}$ rectangle of the partition iff $M_{ij}$=1. Call the closed shaded region $P_0$.\\
{\footnotesize
\begin{figure}[H]
\[ \left( \begin{array}{cccccc}
0 & 0 & 0 & 0 & 0 & 1\\
1 & 1 & 0 & 0 & 0 & 1\\
0 & 1 & 0 & 0 & 0 & 1\\
0 & 0 & 1 & 0 & 0 & 1\\
0 & 0 & 1 & 0 & 1 & 0\\
0 & 0 & 0 & 1 & 1 & 0\\
\end{array} \right) \]

\vspace{1mm}
\hspace{-7mm}
\includegraphics[width=6cm, height=6cm]{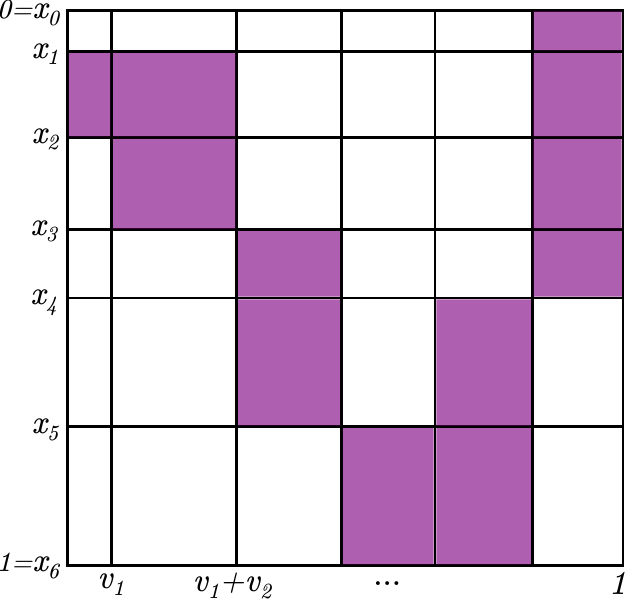}
\caption{The shaded region $P_0$ ``resembles" $M$, with widths of columns given by the entries of the positive eigenvector of $M$, and heights of rows given by the entries of the positive eigenvector of $M^T$, both corresponding to the leading eigenvalue $\lambda$.}
\end{figure}
}

We wish to map the shaded region in the $i^{th}$ column to the shaded region in the $i^{th}$ row, stretching horizontally by $\lambda$ and shrinking vertically by $\lambda$. Widths and heights were chosen as eigenvectors of $M$ and $M^T$ to make this possible. But first, if there are multiple rectangles in a row, we ``draw them together" by identifying their inner edges, as shown in the following figure. We denote by $R_i$ the rectangle we obtain this way from combining the shaded rectangles in the $i^{th}$ row, and by $C_i$ the shaded region in the $i^{th}$ column of $P_0$.\\

\begin{figure}[H]
\includegraphics[scale=1]{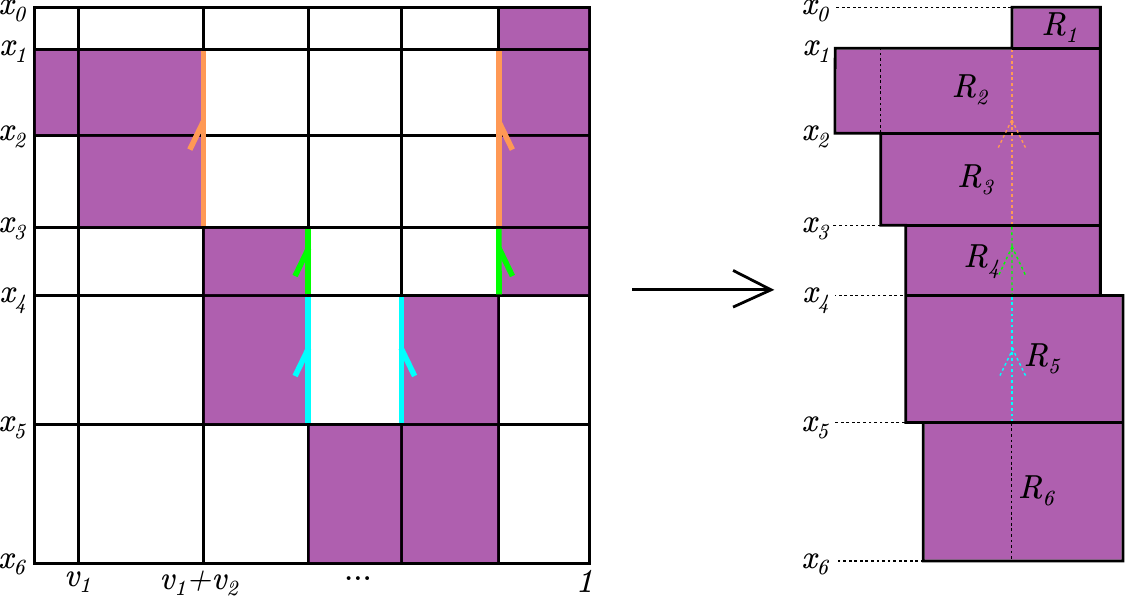}
\caption{We ``draw together" all shaded blocks in each row to get $R_i$, and we align $R_i, R_{i+1}$ on the left or right according to whether $\alpha(i) =-1$ or $1$ (respectively). We keep track of where the $i^{th}$ column of $P_0$, $C_i$, goes under this ``drawing together".}
\end{figure}

Each pair of adjacent rows are aligned either on the left or the right. Here we use our alignment function $\alpha$ to decide where to align: If $\alpha(i)=-1$, align rows $R_i, R_{i+1}$ on the left, otherwise align them on the right.\\

Now define, $f_0:P_0\to P_0$ as follows: On the $i^{th}$ row, $R_i$, stretch vertically by $\lambda$ and shrink horizontally by $\lambda$. Then send this scaled rectangle to the $i^{th}$ column, $C_i$, by a translation preceded by:

{\small
\begin{itemize}
\item The Identity, \hspace{4.6cm} case $\epsilon=1$ and $h'>0$ on $I_i$\\
\item Rotation by $180^\circ$, \hspace{3.9cm} case $\epsilon=1$ and $h'<0$ on $I_i$\\
\item A reflection about a vertical line, \hspace{1.3cm} case $\epsilon=-1$ and $h'>0$ on $I_i$\\
\item A reflection about a horizontal line, \hspace{.9cm} case $\epsilon=-1$ and $h'<0$ on  $I_i$\\
\end{itemize}
}

\begin{figure}[H]
\includegraphics[scale=.9]{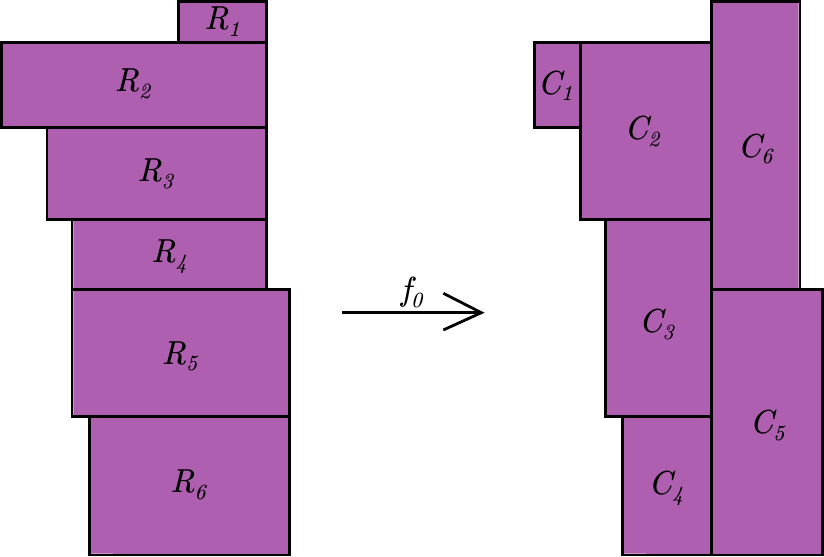}
\caption{We define $f_0$ by sending the $i^{th}$ row of $P_0$, $R_i$, to the $i^{th}$ column, $C_i$, affinely, according to the rules above. In this example, $\epsilon=1$, and $R_1, R_5, R_6$ are rotated by $\pi$ before being scaled and sent to the $C_i$. }
\end{figure}

As it is, $f_0$ is only a relation and not a well-defined function (on $R_i\cap R_{i+1}$). We define an equivalence relation ($\sim$) on the boundary of $P_0$ as follows: If a horizontal edge between the $R_i$ is sent to two different edges, identify the images, and all their iterates. For the identification of vertical edges, if $f_0^{\circ n}(x)=f_0^{\circ n}(y)$, for some $n>0$, identify $x$ to $y$. After the identifications, $f_0$ factors through as a homeomorphism, $f$, of $P:=P_0/\sim$, a polygon with identifications on the boundary.\\

\section{Proof of Theorem 4}
We are now ready to prove our main theorem.\\
\begin{thm*} 
If $\lambda$ is the leading eigenvalue of a nonsingular, aperiodic, odd-block, $\{0,1\}$-matrix $M$, such that $M$ satisfies the \emph{one-sided} condition, and $(M,\epsilon)$ satisfies the \emph{alignment} condition, then\\ 
\begin{itemize}
\item[(i)] $\lambda$ is the dilatation of a pseudo-Anosov map, $\psi$, of a compact surface, $S_g$.\\
\item[(ii)] When $M$ is an $n\times n$ minimally odd-block matrix, $g = \floor{n/2}$, and otherwise $g\le n/2$.\\
\item[(iii)] $\psi$ is orientation preserving (resp. reversing) when $\epsilon$ is 1 (resp. -1). \\
\item[(iv)] When $g=\frac{1}{2}\dim(M)$ a basis for $H_1(S_g)$ may be chosen so that the action of $\psi$ on $H_1(S_g)$ is given by $M$. $\chi(M)$ is palindromic or anti-palindromic in this case (resp. whether $\epsilon = 1$ or $-1$).
\end{itemize}
\end{thm*}

\begin{proof}
Let $M$ and $\epsilon$ {\small$\in\{-1,1\}$} be given and satisfy the conditions of the theorem. Construct the interval map $h$ as in section $2$, and construct $P_0$ and $f_0$ as in section $4$. We will study the identifications $\sim$ induced by $f_0$ on $P_0$ and show that $P=P_0/\sim$ is homeomorphic to a sphere with finitely many cone points.\\

First consider the horizontal edges of the $R_i$. Define $E_i$ to be the horizontal edge at height $x_i$, i.e. $\boldsymbol{E_i :=} \{(x,y)\in[0,1]^2:y=x_i\}\cap P_0$. $E_i$ is sent to $E_{\phi(i)}$.\\

If $x_i$ is a noncritical point of $h$, the interior of $R_i\cup R_{i+1}$ is mapped homeomorphically to its image: part $(b)$ of the alignment condition ensures that the alignments at $E_i$ and $E_{\phi(i)}$ are compatible, so as to make $f_0$ continuous on the intersection $R_i\cap R_{i+1}$, as illustrated in the figure below. Recall part $(b)$ of the alignment condition: \\

For noncritical $x_i, \,\,\alpha(i) = 
\begin{cases}
      \hfill +\epsilon\,\alpha(\phi_i) \hfill & \text{ if $h'(x_i)>0$.}\\
      \hfill -\epsilon\,\alpha(\phi_i) \hfill & \text{ if $h'(x_i)<0$.}
 \end{cases}$

\begin{figure}[H]
\centering
\def\svgwidth{0.75\columnwidth}
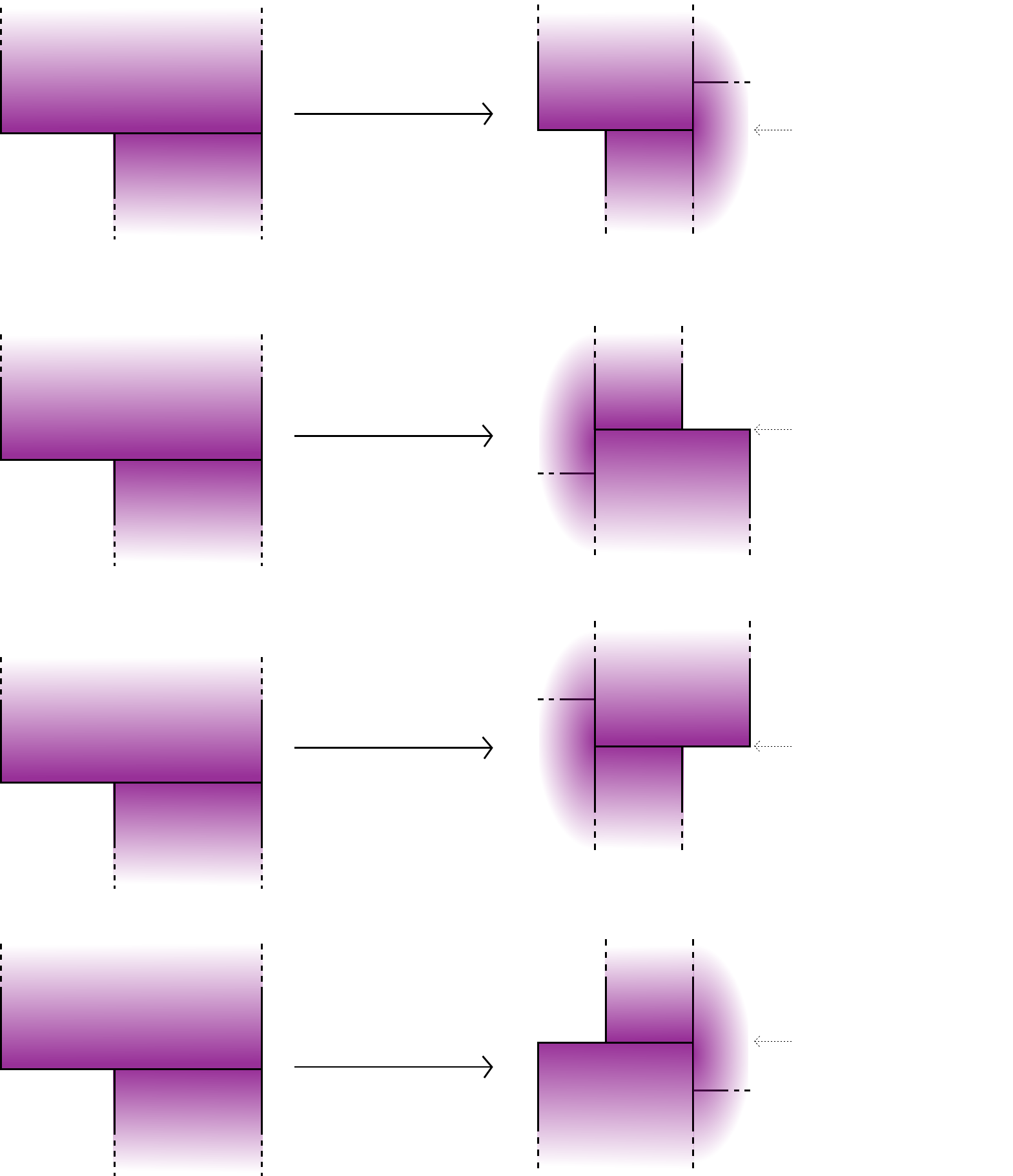
\caption{Illustrating the cases that occur when $x_i$ is noncritical and $\alpha(i)=1$ (i.e. when $R_i$ and $R_{i+1}$ are aligned on the right). Alignment condition $(b)$ ensures that the alignment at $E_{\phi(i)}$ is so that $R_i\cup R_{i+1}$ is mapped homeomorphically whether it's rotated, reflected, or not.}
\end{figure}

If $x_i$ is a critical point of $h$, $E_i$ is sent to the boundary of $P_0$. $R_i\cap R_{i+1}\subset E_i$ is sent to extreme parts of $E_{\phi(i)}\cap \partial P_0$, ($f_0$ is not a well defined function on $R_i\cap R_{i+1}$). Part $(a)$ of the alignment condition ensures that extreme sides of the horizontal edge $E_{\phi(i)}\cap \partial P_0$ are glued, and the condition was set precisely to make this happen. One case is illustrated below.\\

\begin{figure}[H]
\centering
\def\svgwidth{\columnwidth}
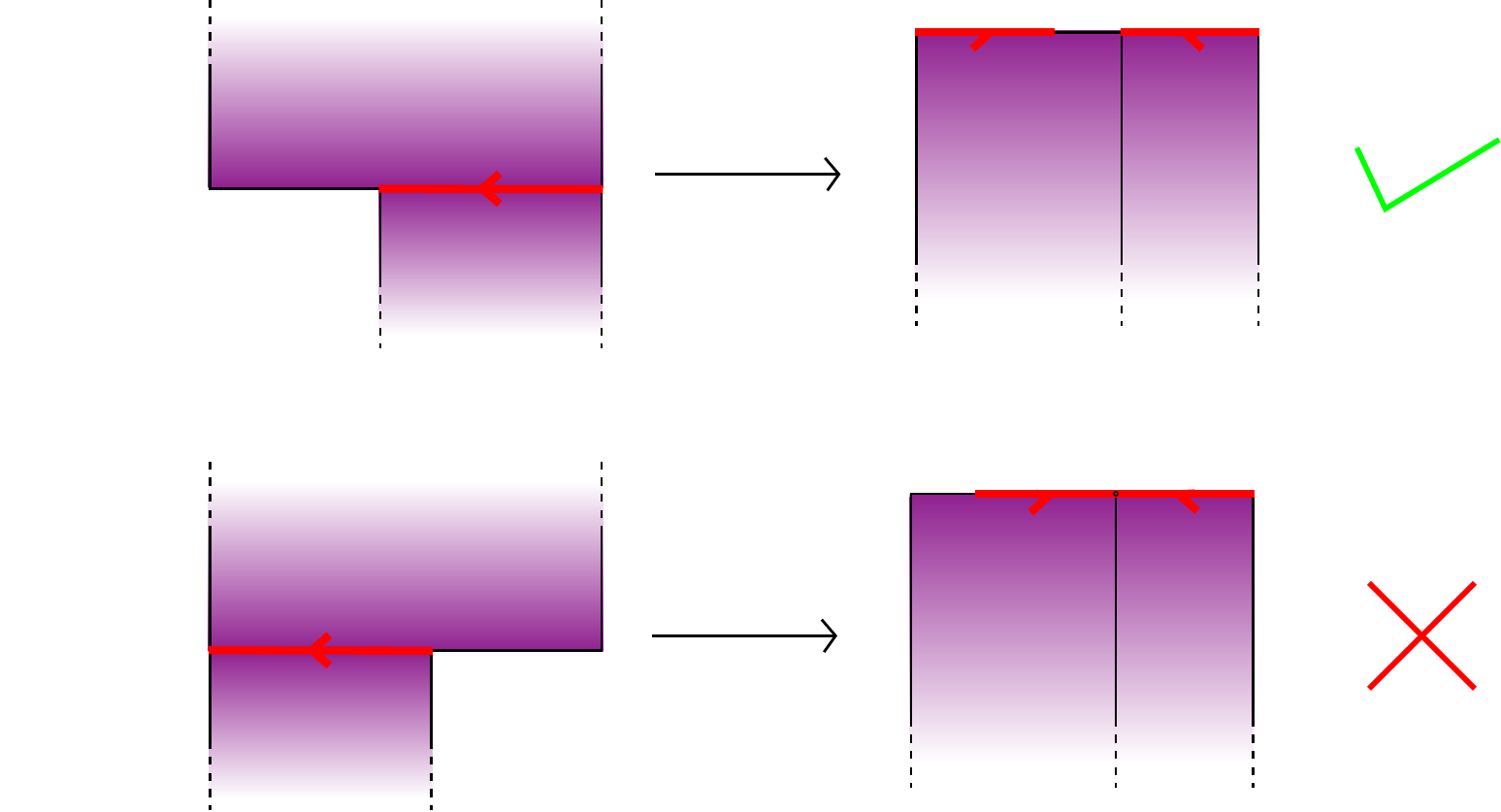
\caption{Illustrating the case when $x_i$ is a critical point of $h$. Here  $\epsilon=1$ and $x_i$ is a local minimum. $\alpha(i)$ must be chosen to be $1$ since we want horizontal edges to have just cone point in the center.}
\end{figure}

\begin{rem}
Note that, as in the figure above, being a local minimum of $h$ means the red edges go to the top of a column rather than the bottom, (since $h$ ``resembles" the vertical \emph{reflection} of $M$).\\
\end{rem} 

In the case of the figure above, $\epsilon=1$ and $x_i$ is a local minimum of $h$, so we want $\alpha(i)$ to be $1$, as it should be if we want a finite type surface, and this is what part $(a)$ of the alignment condition ensures:\\

For critical $x_i, \,\,\alpha(i) = \begin{cases}
      \hfill -\epsilon \hfill & \text{ if $x_i$ is a local max.} \\ 
      \hfill +\epsilon \hfill & \text{ if $x_i$ is a local min.} \\
  		\end{cases}$\\
\vspace{3mm}

Now, let $H_i$ be the horizontal edges of $\partial P_0$, (i.e. $\boldsymbol{H_i :=} \partial P_0\cap E_i$). Identify the left half of $H_i$ with its right half. Thus each $H_i$ will have a cone point of angle $\pi$ at its center; call it $\boldsymbol{Q_i}$.\\

Each $H_i$ is shrunk by $\lambda$ under $f_0$. Moreover, since $(M,\epsilon)$ satisfies the alignment condition, the intervals $f_0(H_i)\subset H_{\phi(i)}$ and $H_{\phi(i)}$ are both centered at $Q_{\phi(i)}$. After folding each horizontal edge $H_i$, $f_0$ is well-defined.\\

\begin{rem} We may assume each $H_i$ has positive width: If two adjacent rows $R_i$ and $R_{i+1}$ had the same width, then $x_{\phi^{-1}(i)}$ could not have been a critical point of $h$. So $R_{\phi^{-1}(i)}\cup R_{\phi^{-1}(i)+1}$ is mapped homeomorphically to $R_i\cup R_{i+1}$, and hence $R_{\phi^{-1}(i)}$ and $R_{\phi^{-1}(i)+1}$ have the same width, so $x_{\phi^{-2}(i)}$ must be noncritical, and so on. Since $\phi$ is a permutation, this implies all iterates of $x_i$ must be noncritical. But we may assume WLOG (see end of section 3) that all cycles in the permutation $\phi$ contain critical points.\\
\end{rem}

Now we turn to the vertical edges of $\partial P_0$. These are stretched by $\lambda$ under $f_0$, and sent to other vertical edges or to the interior of $P_0$. Since the horizontal edges of $\partial P_0$ are folded in half, the vertical edges are connected end-to-end and form a circle, say $\boldsymbol{S}$. i.e. $S$ is the quotient of $\partial P_0$ obtained by collapsing each $H_i$ to a point. Denote the image of $H_i$ in $S$ as $\boldsymbol{a_i}$, $i=0,\dots, n$.\\

For $x_i$ critical, ${\small 1\le i\le n-1}$, for some $c_i>0$, the two arcs starting at $a_i$, say $\boldsymbol{J_i:=}(a_i-c_i,a_i)$ and $\boldsymbol{J'_i:=}(a_i,a_i+c_i)$ are identified and sent to the interior of $P_0$, (see figure $7$). The points $a_i-c_i$ and $a_i+c_i$ are sent to the same point $a_{\phi(i)}$.\\

\begin{figure}[H]
\centering
\def\svgwidth{.6\columnwidth}
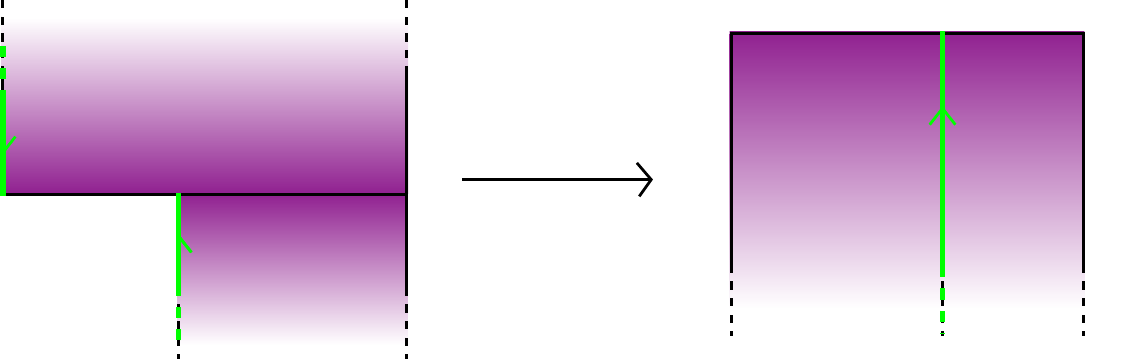
\caption{Illustrating how the vertical edges are glued.}
\end{figure}

For $x_i$ noncritical, $f_0$ does not induce any identifications under one iterate. The vertical edges at $a_i$ are sent to $a_{\phi(i)}$. However, some iterate of $x_i$ is a critical point, and as a result, similar identifications are induced at noncritical $x_i$ (except at $x_0$ and $x_n$). \\

All gluings on vertical edges are induced this way. So the parts that are not identified under any forward iterate form $n-1$ closed intervals. Since $f_0$ stretches by $\lambda>1$, these intervals have to be single points. Moreover, each component of the complement of these $n-1$ points is folded in half, so the $n-1$ points form a single point under identification, call it $\boldsymbol{Q}$. \\

All other points on $\partial P_0$ are regular points (angle = $2\pi$) by inspection. So we have $n+1$ points, $Q_i$, of angle $\pi$ and a single point, $Q$, of angle $(n-1)\pi$. Thus the Euler characteristic of $P=P_0/\sim$ is $2$, i.e. $P$ is a sphere.\\

Taking the double cover of $P$, (ramified at $Q$ and the $Q_i$, when $n$ is even, and only ramified at the $Q_i$ when $n$ is odd), gives us a surface $S_g$ of  genus $g=n/2$ when $n$ is even, and $g=(n-1)/2$ when $n$ is odd. Lift the map $f$ to a map $\psi:S_g\to S_g$, and lift the horizontal and vertical foliations from the plane, to get two transverse measured foliations on $S_g$, invariant under $\psi$.\\

\begin{figure}[H]
\includegraphics[scale=2]{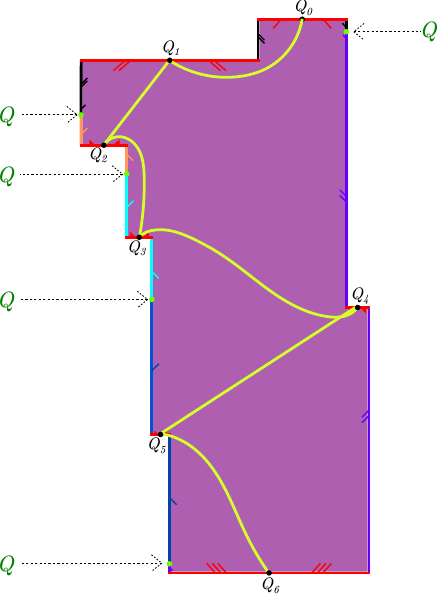}
\caption{The yellow curves, when lifted to the double cover, form closed loops, and give a basis for $H_1(S_g)$.}
\end{figure}

The curve between $Q_i$ and $Q_{i+1}$ in $P_0$ lifts to two curves in the double cover, which together form a closed loop. When $n$ is even, these loops, for $i=0,...,n-1$, form a basis for $H_1(S_g)$. Under this basis, the induced map in homology $\psi_*$ is precisely the matrix $M$.\\
\end{proof}

\section{Acknowledgements}

This paper would not be possible without W. Thurston's ideas to try to construct pseudo-Anosov maps using postcritically-finite maps of intervals, and turning them into $\omega$-limit sets of $2$-dimensional dynamical systems.\\
 
We would also like to thank John H. Hubbard and Dylan Thurston for discussions and helpful suggestions during the process of writing this paper, Joshua P. Bowman, Andre de Carvalho, Giulio Tiozzo and Danny Calegari for helpful insights. 
The first author was partially supported by
the ERC Grant Nb. 10160104.


\begin{thebibliography}{9}
\bibitem{CH} de Carvalho, A. and Hall, T.  (2004). \emph{Unimodal generalized pseudo-Anosov maps.} Geom. Topol., 8, 1127-1188. \\
\bibitem{Farb} Farb, B. (2006). \emph{Some problems on mapping class groups and moduli space.
Problems on mapping class groups and related topics.} Proc. Sympos. Pure Math., 74, 11-55. \\
\bibitem{Fr}  Fried, D. (1985). \emph{Growth rate of surface homeomorphisms and flow equivalence.} Ergod. Theory Dyn. Syst., 5(04), 539-563.\\
\bibitem{Li} Lind, D. A. (1984). \emph{The Entropies of Topological Markov Shifts and a Related Class of Algebraic Integers.} Mathematical Sciences Research Institute.\\
\bibitem{Th} Thurston, W. (2014). \emph{Entropy in dimension one.} arXiv:1402.2008.\\
\end{thebibliography}

\end{document}